\documentclass[12pt]{amsart}
\usepackage{hyperref}

\newtheorem{theorem}{Theorem}[section]
\newtheorem{corollary}[theorem]{Corollary}
\newtheorem{lemma}[theorem]{Lemma}
\newtheorem{proposition}[theorem]{Proposition}

\theoremstyle{theorem}
\newtheorem{innercustomthm}{Theorem}
\newenvironment{customthm}[1]
  {\renewcommand\theinnercustomthm{#1}\innercustomthm}
  {\endinnercustomthm}

\theoremstyle{definition}

\theoremstyle{remark}
\newtheorem{remark}[theorem]{Remark}

\numberwithin{equation}{section}

\def\no{\nonumber}

\newcommand{\p}{\partial}

\newcommand{\fder}[2]{\frac{\partial #1}{\partial #2}}

\begin{document}
\title{Some results on higher eigenvalue optimization}
\author{Ailana Fraser}
\address{Department of Mathematics \\
                 University of British Columbia \\
                 Vancouver, BC V6T 1Z2}
\email{afraser@math.ubc.ca}
\author{Richard Schoen}
\address{Department of Mathematics \\
                 University of California \\
                 Irvine, CA 92617}
\email{rschoen@math.uci.edu}
\thanks{
A. Fraser was partially supported by  the 
Natural Sciences and Engineering Research Council of Canada and R. Schoen
was partially supported by NSF grant DMS-1710565. Part of this work was done while the authors were visiting the Institute for Advanced Study, with funding from NSF grant DMS-1638352 and the James D. Wolfensohn Fund, and the authors gratefully acknowledge the support of the IAS}

\begin{abstract}
In this paper we obtain several results concerning the optimization of higher Steklov eigenvalues both in two and higher dimensional cases. We first show that the normalized (by boundary length) $k$-th Steklov eigenvalue on the disk is not maximized for a smooth metric on the disk for $k\geq 3$. For $k=1$ the classical result of \cite{W} shows that $\sigma_1$ is maximized by the standard 
metric on the round disk. For $k=2$ it was shown \cite{GP1} that $\sigma_2$ is not maximized
for a smooth metric. We also prove a local rigidity result for the critical catenoid and the critical
M\"obius band as free boundary minimal surfaces in a ball under $C^2$ deformations. We next
show that the first $k$ Steklov eigenvalues are continuous under certain degenerations of 
Riemannian manifolds in any dimension. Finally we show that for $k\geq 2$ the supremum of
the $k$-th Steklov eigenvalue on the annulus over all metrics is strictly larger that that over $S^1$-invariant 
metrics. We prove this same result for metrics on the M\"obius band.
\end{abstract}

\maketitle

\section{Introduction}
In this paper we obtain several results concerning the optimization of higher Steklov eigenvalues both in two and higher dimensional cases. Recall that for a compact Riemannian manifold with non-empty boundary we have the Steklov spectrum which consists of the eigenvalues of the Dirichlet to Neumann map. We denote these eigenvalues $\sigma_0=0<\sigma_1\leq\sigma_2\ldots$ and they form an infinite discrete sequence tending to infinity. A Steklov eigenfunction $u$ with eigenvalue $\sigma$ is then a non-zero solution of $\Delta u=0$ in $M$ with $\frac{\partial u}{\partial \nu}=\sigma u$ on $\partial M$ where $\nu$ denotes the outward unit normal to $\partial M$.

A classical result of J. Hersch, L. Payne, and M. Schiffer \cite{HPS} from $1975$ gives the upper bound $\sigma_k\cdot L(\partial D)\leq 2\pi k$ for all  metrics on the disk $D$ and for all $k\geq 1$. In $2010$ it was shown by A. Girouard and I. Polterovich \cite{GP1} that this bound is sharp for all $k$ but is {\it not} attained by a smooth metric on the disk for $k=2$. The bound and the result that it is attained by the standard round disk for $k=1$ is a classical result of R. Weinstock \cite{W}. In Section 2 of this paper we extend the result of \cite{GP1} to show that the bound is not attained for a smooth metric for all $k\geq 2$. The proof is based on our earlier work \cite{FS3} on uniqueness of free boundary minimal disks in higher dimensions together with the characterization of maximizing metrics given in \cite{FS2}.

In Section 3 we prove a local uniqueness theorem among free boundary minimal surfaces for the critical catenoid in $\mathbb{B}^n$ and for the critical M\"obius band in $\mathbb{B}^n$. It is not known whether there are other embedded free boundary minimal annuli besides the critical catenoid in $\mathbb{B}^n$, but we are able to show that there are none which lie in a $C^2$ neighborhood of the critical catenoid except rotates of the critical catenoid. We prove an analogous result for the critical M\"obius band. These results are consequences of the work of \cite{FS4} where it is shown that the critical catenoid is the only free boundary minimal annulus with the coordinate functions being first Steklov eigenfunctions. It is also shown in \cite{FS4} that the critical M\"obius band is the only free boundary minimal M\"obius band with coordinate functions being first Steklov eigenfunctions.

In Section 4 we consider the question of the degenerations of Riemannian manifolds under which the first $k$ Steklov eigenvalues are continuous. This question is important when one attempts to construct metrics which optimize an eigenvalue. We prove the following result which concerns the case in which a manifold degenerates into a disjoint union of manifolds.

\begin{theorem} \label{theorem:gluing-boundary}
Let $M_1, \ldots, M_s$ be compact n-dimensional Riemannian manifolds with nonempty boundary. Given $\epsilon >0$, there exists a Riemannian manifold $M_\epsilon$, obtained by appropriately gluing $M_1, \ldots, M_s$ together along their boundaries, such that 
\begin{align*}
       \lim_{\epsilon \rightarrow 0} | \partial M_\epsilon| &=|\partial (M_1 \sqcup \cdots \sqcup M_s)| 
       \quad \mbox{ and} \\
       \lim_{\epsilon \rightarrow 0} \sigma_k (M_\epsilon) 
       &=\sigma_k(M_1 \sqcup \cdots \sqcup M_s)
\end{align*}
for $k=0, \, 1,\, 2, \ldots$.
\end{theorem}

The results of \cite{GP1} may be considered as a very special case (gluing copies of the unit disk), and as is discussed there, the shape of the neck which is used in the gluing is a delicate consideration (see also \cite{K2}). This is slightly different in the cases $n=2$ and $n\geq 3$. In the case $n=2$ we use essentially a rectangular neck of approximately equal side and vanishingly small side lengths, while for $n\geq 3$ we use a portion of a catenoidal hypersurface in order to avoid concentration of eigenfunctions on the neck region. There is a substantial amount of delicate analysis involved in giving a rigorous proof of this result. 

We also consider the result of interior gluings such as connected sums with small necks. In this case we prove under quite weak conditions on the neck region the result.

\begin{theorem} \label{theorem:gluing-interior}
Let $M_1, \ldots, M_s$ be compact n-dimensional Riemannian manifolds with nonempty boundary. Given $\epsilon >0$ there exists a Riemannian manifold $M_\epsilon$, obtained by appropriately gluing $M_1, \ldots, M_s$ together along there interiors, such that 
$\partial M_\epsilon=\partial (M_1 \sqcup \ldots \sqcup M_s)$ and
\[
       \lim_{ \epsilon \rightarrow 0} \sigma_k (M_\epsilon) =  \sigma_k(M_1 \sqcup \cdots \sqcup M_s)
\]
for $k=0, \, 1,\, 2, \ldots$.
\end{theorem}
The fact that the shape of the neck is unimportant in this theorem is consistent with the recent results of B. Colbois, A. Girouard, and A. Hassannezhad \cite{CGH} which show that up to constants the Steklov eigenvalues depend only on the geometry near the boundary of a manifold.

The combination of these results in the case $n=2$, which is stated in Corollary \ref{pet_cor}, yields the bounds stated for the supremum of the $k$-th Steklov eigenvalue of a surface in the paper of R. Petrides \cite{P2}.

Finally in Section 5 of this paper we explore the question of maximizing eigenvalues with symmetry imposed on the competing metrics versus maximizing over all smooth metrics. We consider this question in two specific cases of surfaces with $S^1$ symmetry group. The first case is the annulus where one can pose the maximization question over $S^1$-invariant metrics or over all metrics. In the case of the annulus we showed in our earlier work \cite{FS4} that for $k=1$ the global maximizer is $S^1$-invariant, so these maxima are the same. For $\sigma_k$ with $k\geq 2$ we show that the supremum over all metrics is strictly larger than the supremum over $S^1$-invariant metrics. It was shown by X. Q. Fan, L. F. Tam, and G. Yu \cite{FTY} that for $S^1$-invariant metrics all $\sigma_k$ for $k\neq 2$ are maximized by a smooth $S^1$-invariant metric. In the case $k=2$ it is
possible for a sequence of $S^1$-invariant metric annuli to degenerate to a pair of disks which is an explanation for why the extremal metric does not exist for $k=2$. On the other hand we show that for metrics degenerating to the union of the critical catenoid and a disk the limiting value is larger, so the $S^1$-invariant degeneration is not optimal among general metrics on the annulus.

The second case we consider is the case of the M\"obius band with $S^1$ symmetry. In this case it was recently shown by A. Fraser and P. Sargent \cite{FSa} that there is a smooth $S^1$-invariant metric which maximizes $\sigma_k$ for each $k$. In our earlier paper \cite{FS4} we showed that for $k=1$ the maximizer over all metrics exists and is $S^1$-invariant. We show here for $k\geq 2$ the supremum of $\sigma_k$ over all smooth metrics on the M\"obius band is strictly larger than the supremum over $S^1$-invariant metrics.

\section{Simply connected surfaces} \label{section:disk}

In this section we show that if $M$ is a simply connected surface with boundary, then for $k \geq 2$, the supremum of the $k$-th nonzero normalized Steklov eigenvalue $\sigma_k(g)L_g(\partial M)$ over all smooth metrics on $M$ is not achieved. There are two main ingredients in the proof. The first is the following characterization of maximizing metrics.

\begin{proposition}[{\cite[Proposition 2.4]{FS2}}] \label{prop:extremal}
If $M$ is a surface with boundary, and $g_0$ is a metric on $M$ with
\[
      \sigma_k(g_0)L_{g_0}(\p M)=\max_g \sigma_k(g) L_g(\p M)
\] 
where the max is over all smooth metrics on $M$. Then, rescaling the metric such that $\sigma_k(g_0)=1$, there exist independent $k$-th eigenfunctions $u_1, \ldots, u_n$, for some $n \geq 2$, that give a proper conformal immersion $u=(u_1, \ldots, u_n): M \rightarrow \mathbb{B}^n$ that is an isometry on $\p M$; in particular, $u(M)$ is a free boundary minimal surface.
\end{proposition}

The second ingredient is the following minimal surface uniqueness theorem.

\begin{theorem}[{\cite[Theorem 2.1]{FS3}}] \label{theorem:uniqueness-disk}
Let $u: D \rightarrow \mathbb{B}^n$ be a proper branched minimal immersion, such that $u(D)$ meets $\p \mathbb{B}^n$ orthogonally. Then $u(D)$ is an equatorial plane disk.
\end{theorem}

We now state the theorem:

\begin{theorem} \label{theorem:disk}
Let $M$ be a simply connected surface with boundary. For $k \geq 1$, for any smooth metric $g$ on $M$, 
\[
         \sigma_k(g) L_g(\partial M) \leq 2 \pi k.
\]
For $k=1$, the equality is achieved if and only if $g$ is $\sigma$-homothetic to the Euclidean unit disk. For $k \geq 2$ the inequality is strict, and equality is achieved in the limit by a sequence of metrics degenerating to a union of $k$ touching Euclidean unit disks.
\end{theorem}

\begin{proof}
The case $k=1$ is due to Weinstock \cite{W}. For $k \geq 2$ the upper bound $\sigma_k(g) L_g(\partial M) \leq 2 \pi k$ is due to Hersch-Payne-Schiffer \cite{HPS}. Girouard-Polterovich \cite{GP1} proved that this upper bound is sharp; precisely, they show that the upper bound is achieved in the limit by a sequence of metrics degenerating to a union of $k$ touching Euclidean unit disks. Moreover, for $k=2$ Girouard-Polterovich \cite{GP1} proved that the inequality is strict. We now show that the inequality is strict for all $k \geq 2$.

Suppose there exists a smooth metric $g$ such that $\sigma_k(g)L_g(\partial M)=2\pi k$. Since $\sigma_kL$ is invariant under rescaling of the metric, without loss of generality, assume $\sigma_k(g)=1$. Then by Proposition \ref{prop:extremal} there exist $k$-th eigenfunctions $u_1, \ldots, u_n$, for some $n \geq 2$, such that 
\[
       u:=(u_1, \ldots, u_n): M \rightarrow \mathbb{B}^n
\]
is a proper conformal branched minimal immersion such that $u(M)$ meets $\partial \mathbb{B}^n$ orthogonally, and $g$ is the induced metric on $\partial M$. By Theorem \ref{theorem:uniqueness-disk}, $u(M)$ is an equatorial plane disk. Thus, $g$ is $\sigma$-homothetic (see \cite[Definition 2.1]{FS4}) to the induced metric on the Euclidean unit disk $\mathbb{D}$, and so $\sigma_k(g)L_g(\partial M)= \sigma_k(\mathbb{D})L(\partial \mathbb{D})$. But $\sigma_k(\mathbb{D})L(\partial \mathbb{D})<2\pi k$, a contradiction.
\end{proof}

\section{Rigidity of the critical catenoid and M\"obius band} \label{section:rigidity}

The next natural case to consider after the disk is the annulus. In \cite{FS4} the authors proved that there exists a smooth metric that maximizes the first nonzero normalized Steklov on the annulus. Moreover, the authors proved that any maximizing metric on the annulus is $\sigma$-homothetic (see \cite[Definition 2.1]{FS4}) to the induced metric on the `critical catenoid'. The critical catenoid is the unique portion of a suitably scaled catenoid which defines a free boundary surface in $\mathbb{B}^3$.

\begin{theorem}[{\cite[Theorem 1.3]{FS4}}] \label{theorem:annulus}
For any metric on the annulus $M$ we have 
\[
         \sigma_1L\leq (\sigma_1L)_{cc}
\]
with equality if and only if $M$ is $\sigma$-homothetic to the critical catenoid.
\end{theorem}

For higher eigenvalues for the annulus there are upper bounds due to Karpukhin \cite{K1} (see also \cite{GP2}),
\[
        \sigma_k(g) L_g(\partial M) \leq 2\pi(k+1),
\]
but it is an open question whether these are sharp upper bounds, and whether there exist maximizing metrics for the higher eigenvalues. For the disk, the nonexistence of metrics that maximize higher eigenvalues, Theorem \ref{theorem:disk}, uses the minimal surface uniqueness theorem, Theorem \ref{theorem:uniqueness-disk}. For the annulus, if there exists a metric that maximizes $\sigma_kL$, then Proposition \ref{prop:extremal} characterizes the maximizing metric as being $\sigma$-homothetic to the induced metric from a free boundary minimal immersion of the annulus into $\mathbb{B}^n$ by $k$-th eigenfunctions, for some $n \geq 2$. Although the critical catenoid is the only known free boundary minimal annulus in $\mathbb{B}^3$, there are many other known free boundary minimal annuli in $\mathbb{B}^4$ \cite{FTY}, \cite{FSa}. The explicit characterization of the metric that maximizes $\sigma_1L$ in Theorem \ref{theorem:annulus} uses the following minimal surface uniqueness theorem that characterizes the critical catenoid as the only free boundary minimal immersion of the annulus into $\mathbb{B}^n$ by {\em first} eigenfunctions.

\begin{theorem}[{\cite[Theorem 1.2]{FS4}}] \label{theorem:uniqueness-cc}
If $\Sigma$ is a free boundary minimal surface in $\mathbb{B}^n$ which is homeomorphic
to the annulus and such that the coordinate functions are first Steklov eigenfunctions, then $n=3$ and 
$\Sigma$ is congruent to the critical catenoid.
\end{theorem}

A consequence of Theorem \ref{theorem:uniqueness-cc} is the following local rigidity result for the critical catenoid.

\begin{theorem} \label{theorem:annulus-rigidity}
Any free boundary minimal annulus in $\mathbb{B}^n$ that is sufficiently $C^2$-close to the critical catenoid is a rotation of the critical catenoid.
\end{theorem}

\begin{proof}
Let $\Sigma$ be the critical catenoid, and suppose that $\tilde{\Sigma}$ is a free boundary minimal annulus in $\mathbb{B}^n$ that is $C^2$ close to $\Sigma$. We know that $\sigma_0(\Sigma)=0$, $\sigma_1(\Sigma)=\sigma_2(\Sigma)=\sigma_3(\Sigma)=1$, and $\sigma_4(\Sigma) > 1$. Now $\sigma_0(\tilde{\Sigma})=0$, and since $\tilde{\Sigma}$ is a free boundary minimal surface, the coordinate functions $x^1, \ldots, x^n$ in $\mathbb{R}^n$ restricted to $\tilde{\Sigma}$ are Steklov eigenfunctions with eigenvalue 1. Note that the Steklov spectrum varies continuously if we take a $C^2$ perturbation of $\Sigma$,  \cite[Lemma 2.5]{FS2}. Therefore, given $\epsilon>0$, if $\tilde{\Sigma}$ is sufficently $C^2$-close to $\Sigma$, then $|\sigma_k(\tilde{\Sigma})-\sigma_k(\Sigma)|<\epsilon$. Choosing $\epsilon$ small, this implies that $n=3$, and $\sigma_1(\tilde{\Sigma})=1$. Therefore, by Theorem \ref{theorem:uniqueness-cc}, $\tilde{\Sigma}$ is congruent to $\Sigma$, and hence is a rotation of $\Sigma$.
\end{proof}

We have a similar local rigidity result for the critical M\"obius band. The {\em critical M\"obius band} is an explicit free boundary minimal embedding of the M\"obius band into $\mathbb{B}^4$ by first Steklov eigenfunctions (see \cite[Section 7]{FS4}). In \cite[Theorem 1.5]{FS4} the authors proved that the induced metric on the critical M\"obius band uniquely (up to $\sigma$-homothety) maximizes the first normalized Steklov eigenvalue among all smooth metrics on the M\"obius band. As in the case of the annulus, the characterization of the maximizing metric uses a minimal surface uniqueness theorem, \cite[Theorem 7.4]{FS4}, showing that the critical M\"obius band is the unique free boundary minimal M\"obius band in $\mathbb{B}^n$ such that the coordinate functions are first Steklov eigenfunctions. Another consequence of this is the following local uniqueness theorem for the critical M\"obius band:

\begin{theorem} \label{theorem:mobius-rigidity}
Any free boundary minimal M\"obius band in $\mathbb{B}^n$ that is sufficiently $C^2$-close to the critical M\"obius band is a rotation of the critical M\"obius band.
\end{theorem}

The proof is exactly analogous to the proof of Theorem \ref{theorem:annulus-rigidity}.

\section{Continuity of Steklov eigenvalues under degenerations} \label{section:gluing}

In this section we prove our main results showing that the first $k$ Steklov eigenvalues are continuous under certain degenerations. The difficult case is that of degenerations along the boundary.

\begin{customthm}{\ref{theorem:gluing-boundary}}
Let $M_1, \ldots, M_s$ be compact n-dimensional Riemannian manifolds with nonempty boundary. Given $\epsilon >0$, there exists a Riemannian manifold $M_\epsilon$, obtained by appropriately gluing $M_1, \ldots, M_s$ together along their boundaries, such that 
\begin{align*}
       \lim_{\epsilon \rightarrow 0} | \partial M_\epsilon| &=|\partial (M_1 \sqcup \cdots \sqcup M_s)| 
       \quad \mbox{ and} \\
       \lim_{\epsilon \rightarrow 0} \sigma_k (M_\epsilon) 
       &=\sigma_k(M_1 \sqcup \cdots \sqcup M_s)
\end{align*}
for $k=0, \, 1,\, 2, \ldots$.
\end{customthm}

We also prove an analogous result in the case of interior degenerations in Theorem \ref{theorem:gluing-interior}.
We remark that there are similar results of this type for closed manifolds (\cite{A}, \cite[Lemma 3.2]{CE}), but the proofs are technically quite different. 

\subsection{Preliminaries} \label{section:preliminaries}
Here we collect various useful estimates on domains in a manifold with bounded geometry. This includes some extensions and refinements of results in \cite[Section 2.2]{ST}. First, it is observed that bounded geometry implies the metric is uniformly equivalent to the Euclidean metric \cite{E} in balls of fixed radius. 

In this work we will need slight modifications of the standard Poincar\'e and Sobolev inequalities for functions in an annulus. We also will need estimates in half balls and half annuli. We use the notation $B^+_r$ to denote the points of $B_r$ which lie in a half space, say $x_n\geq 0$. We let $A^+=B^+_{r_0}\setminus B^+_{r_1}$ in $\mathbb R^n$. We let $\Gamma_r$ denote the portion of $\partial B^+_r$ on which $x_n=0$.

We assume that we have an annulus $A=B_{r_0}\setminus B_{r_1}$ in $\mathbb R^n$ with a metric $g$ which is uniformly equivalent to the Euclidean metric; specifically for a positive constant $C_1$ and all $a\in\mathbb R^n$
\[ 
           C_1^{-1}\sum_{i=1}^na_i^2\leq\sum_{i,j=1}^ng^{ij}a_ia_j\leq C_1\sum_{i=1}^na_i^2.
\] 
Then the following estimates hold. 

\begin{lemma}. \label{inequality} 
Suppose we have an annulus as above. 
\begin{enumerate}
\item \label{poincare} For any smooth function $f$ 
with $f=0$ on the inner boundary $\partial B_{r_1}$, there is a constant depending only on $r_0$, $r_1$ and $C_1$ such that,
\[ 
     \int_Af^2\ dv\leq c\int_A|\nabla_{g}f|^2\ dv.
\]
We also have 
\[ 
      \int_{\Gamma_{r_0}\setminus\Gamma_{r_1}}f^2\ da+\int_{A^+}f^2\ dv\leq c\int_{A^+}|\nabla_{g}f|^2\ dv
\] 
\item \label{sobolev} Assume $n\geq 3$. For any smooth function $f$ on $A$ with $f=0$ on $\partial B_{r_0}$,
there is a constant depending only on $n$ and $C_1$ (independent of $r_0$ and $r_1$) such that, 
\[
      \left(\int_A f^{\frac{2n}{n-2}}\ dv\right)^{\frac{n-2}{n}} \leq c\int_A|\nabla_g f|^2\ dv.
\]
Under the condition that $f=0$ on $\partial A^+\cap \partial B_{r_0}$ we have
\[ 
      \left(\int_{A^+} f^{\frac{2n}{n-2}}\ dv\right)^{\frac{n-2}{n}} \leq c\int_{A^+}|\nabla_g f|^2\ dv.
\]
\end{enumerate}
\end{lemma}

\begin{proof} 
For the  Poincar\'e inequalities, it is noted that for the Euclidean case the constant is the inverse of the lowest eigenvalue for the problem with Dirichlet condition on the inner boundary and Neumann or Steklov conditions on the outer boundary components. Because the metric $g$ is uniformly equivalent to the Euclidean metric, each term of the inequality only varies within multiplicative bounds determined by that equivalence, so the result follows.  

The first version of the Sobolev inequality follows in a standard way from the corresponding $L^1$ inequality
\[ 
     \left(\int_Af^{\frac{n}{n-1}}\ dv\right)^{\frac{n-1}{n}}\leq c\int_A|\nabla_g f|\ dv
\]
for functions $f$ which vanish on the outer boundary. That, in turn, is equivalent to the isoperimetric inequality,
\[ 
              Vol(\Omega)\leq cVol(\partial\Omega\setminus\partial B_{r_1})^{\frac{n}{n-1}}
\]
for any $\Omega\subseteq A$. Note that it suffices to prove the inequality for the Euclidean metric since both sides have bounded ratio (with bound depending on $C_1$) with the corresponding quantity for the metric $g$. We note that the standard isoperimetric inequality for $\Omega$ may be written
\[ 
       Vol(\Omega) \leq c\Big(Vol(\partial\Omega\setminus\partial B_{r_1})
         +Vol(\partial\Omega\cap\partial B_{r_1})\Big)^{\frac{n}{n-1}}.
\]  
Next we observe that the radial projection map $P:A\to\partial B_{r_1}$ given by
$P(x)=r_1 x/|x|$ reduces volumes of hypersurfaces. It thus follows that
\[ 
      Vol(P(\partial\Omega\setminus\partial B_{r_1}))\leq Vol(\partial\Omega\setminus \partial B_{r_1}).
\]
On the other hand any ray through a point of $\partial\Omega\cap\partial B_{r_1}$
must intersect $\partial\Omega$ at a second point, and so we have 
\[ 
    \partial\Omega\cap\partial B_{r_1}\subseteq P(\partial\Omega\setminus\partial B_{r_1}).
\]
Combining this information with the isoperimetric inequality we have
\[ 
     Vol(\Omega)\leq 2^{\frac{n}{n-1}}cVol(\partial\Omega\setminus\partial B_{r_1})^{\frac{n}{n-1}}.
\]
This completes the proof of the desired isoperimetric inequality and the first part of assertion (\ref{sobolev})
follows as indicated above.

To handle the half annulus case we can extend $f$ from $A^+$ to $A$ by even reflection so that both integrals are doubled and we obtain from the previous inequality
\[ 
    \left(\int_{A^+} f^{\frac{2n}{n-2}}\ dv\right)^{\frac{n-2}{n}} \leq c2^{\frac{2}{n}}\int_{A^+}|\nabla_g f|^2\ dv
\]
where $c$ is the previous constant.
\end{proof}

We also need the following version of a logarithmic cut-off function argument.
\begin{lemma} \label{logcutoff} 
Suppose $B_{r_0}$, a ball in $\mathbb R^n$, is equipped with a metric equivalent to the Euclidean metric. For any $\epsilon$, there are small $\rho$, $\rho_1$ with $\rho<\rho_1 \ll r_0$ and a smooth cut-off function $\zeta$, which is $0$ for $x$ in $B_{r_0} \setminus B_{\rho_1}$ and $1$ for $x$ in $B_{\rho}$, such that the following holds. For any smooth function $u$,  
\begin{align*}
     \int_{B_{\rho_1}\setminus B_{\rho}}|\nabla\zeta|^2u^2\ dv 
     &\leq c\epsilon \left( \int_{B_{r_0}\setminus B_\rho} u^2 \ dv + \int_{B_{r_0}} |\nabla u|^2\ dv \right) \\
     \int_{B_{\rho_1}\setminus B_{\rho}}u^2\ dv 
     &\leq c\epsilon \left( \int_{B_{r_0}\setminus B_\rho} u^2 \ dv + \int_{B_{r_0}} |\nabla u|^2 \ dv \right).
\end{align*}
Here $c$ is a constant depending on $r_0$ and bounds on the eigenvalues of the metric with respect to the Euclidean metric. 
\end{lemma} 

\begin{proof}
For $n\geq 3$ we can take $\rho_1=\sqrt{\rho}$, and let $A=B_{\sqrt{\rho}}\setminus B_{\rho}$ and we set
\[ 
          \zeta(r)=\frac{\log(r/\sqrt{\rho})}{\log(1/\sqrt{\rho})}\ \mbox{for}\ \rho\leq r\leq \sqrt{\rho}.
\]
Note that the function $\zeta$ we have chosen is not smooth but only Lipschitz continuous. It is a standard argument to see that such a $\zeta$ can be approximated by smooth functions in the $W^{1,2}$ norm so that we can justify this choice. Also, it suffices to prove the first inequality because, for our choice of $\zeta$,
$1<|\nabla \zeta|$ on $A$.\\

Since $n\geq 3$ we can use the H\"older inequality to obtain
\[ 
        \int_{A}|\nabla\zeta|^2u^2\ dv
        \leq \left(\int_{A} |\nabla\zeta|^n\ dv\right)^{\frac{2}{n}} \left(\int_{A}u^{\frac{2n}{n-2}}\ dv\right)^{\frac{n-2}{n}}.
\]
From the definition of $\zeta$ and the conditions on the metric on the annulus we have
\[ 
         \int_{A}|\nabla\zeta|^n\ dv
         \leq c|\log(\rho)|^{-n}\int_{\rho}^{\sqrt{\rho}} r^{-1}dr\leq c|\log(\rho)|^{1-n}.
\]
Thus for any $\epsilon>0$, when $\rho$ is small enough we have
\[ 
      \int_{B_{\sqrt{\rho}}\setminus B_\rho}|\nabla\zeta|^2u^2\ dv
      \leq \epsilon \left(\int_{B_{\sqrt{\rho}}\setminus B_\rho}u^{\frac{2n}{n-2}}\ dv\right)^{\frac{n-2}{n}}.
\]
Now if $\psi$ is a cut-off function, which is $1$ on $B_{\sqrt{\rho}}$ and supported in $B_{r_0}$, then
we have
\[ 
          \left(\int_{B_{\sqrt{\rho}}\setminus B_\rho}u^{\frac{2n}{n-2}}\ dv\right)^{\frac{n-2}{n}}
          \leq \left(\int_{B_{r_0}}(\psi u)^{\frac{2n}{n-2}}\ dv\right)^{\frac{n-2}{n}}
          \leq c\int_{B_{r_0}}|\nabla (\psi u)|^2\ dv.
\]
Here we have used the Sobolev inequality, Lemma \ref{inequality}(\ref{sobolev}), for functions vanishing on the outer boundary of the annulus $B_{r_0}\setminus B_\rho$. Since the gradient of $\psi$ is bounded we obtain
\[ 
        \left(\int_{B_{\sqrt{\rho}}\setminus B_\rho}u^{\frac{2n}{n-2}}\ dv\right)^{\frac{n-2}{n}}
        \leq c\int_{B_{r_0}}|\nabla (\psi u)|^2\ dv      
        \leq c \left( \int_{B_{r_0} \setminus B_{\sqrt{\rho}}} u^2 \ dv + \int_{B_{r_0}} |\nabla u|^2 \ dv \right).
\]
Combining with our previous inequality we obtain,
\[ 
     \int_{B_{\sqrt{\rho}}\setminus B_{\rho}}|\nabla\zeta|^2u^2\ dv
     \leq c\epsilon \left( \int_{B_{r_0} \setminus B_{\sqrt{\rho}}} u^2 \ dv + \int_{B_{r_0}} |\nabla u|^2 \ dv \right).
\]

For $n=2$ we can obtain the conclusion in a slightly different way. We let $t=\log(\log(1/r))$, with $t_0=\log(\log(1/\rho))$, and choose $\rho_1$ such that $t_0/2=\log(\log(1/\rho_1))$. We now choose $\zeta$ to be a linear function of $t$ which is $1$ at $t=t_0$ and $0$ at $t=t_0/2$. We then have 
\[  
          \int_{B_{\rho_1}\setminus B_\rho}|\nabla\zeta|^2u^2\ dv
          = ct_0^{-2} \int_{B_{\rho_1}\setminus B_\rho}(r\log(1/r))^{-2}u^2\ dv.
\]
We observe that since the metric is near Euclidean in an appropriate annulus $B_{r_0}\setminus
B_\rho$ where $r_0$ is a fixed radius, we may do the estimate in the Euclidean metric. In this case, the volume form $(|x|\log(1/|x|))^{-2}dx^1dx^2$ is that of the hyperbolic metric on the cylinder $\mathbb R\times \mathbb S^1$ given by $dt^2+e^{-2t}d\theta^2$ with coordinates $t=\log(\log(1/|x|))$ and the polar coordinate $\theta$. The annulus now becomes the cylinder $[t_0/2, t_0]\times \mathbb S^1$.
 
Consider the eigenvalue problem with boundary conditions which are Dirichlet at $t=\log\log(1/r_0)$ and Neumann at $t=t_0$. If $g$ denotes the hyperbolic metric we have $\Delta_g(t)=-1$, and so if $f$ is a function which is zero at $t=\log\log(1/r_0)$ we have
\[ 
        \int_{B_{r_0}\setminus B_\rho} f^2\ d\mu_g
        =-\int_{B_{r_0}\setminus B_\rho} f^2\Delta_g(t)\ d\mu_g
        \leq \int_{B_{r_0}\setminus B_\rho} \langle \nabla t,\nabla f^2\rangle\ d\mu_g
\]
where we have used the fact that the boundary term on the outer boundary is nonpositive.
Using the fact that $|\nabla t|=1$ together with the Schwarz inequality we obtain
\[ 
     \int_{B_{r_0}\setminus B_\rho} f^2\ d\mu_g
     \leq 2\int_{B_{r_0}\setminus B_\rho} |f||\nabla_g f|\ d\mu_g.
\]
Using the Schwarz inequality again and the arithmetic mean - geometric mean inequality we obtain the Poincar\'e inequality
\[ 
       \int_{B_{r_0}\setminus B_\rho} f^2\ d\mu_g \leq 4 \int_{B_{r_0}\setminus B_\rho}|\nabla f|^2\ dv,
\]
where we have used the conformal invariance of the Dirichlet integral.

Choosing $\psi$ to be a cut-off function of $t$ which is $0$ for $t\leq\log\log(1/r_0)$ and $1$ for $t\geq\log\log(1/r_0)+1$, we may apply the above Poincar\'e inequality to obtain,
\[ 
    \int_{B_{\rho_1}\setminus B_\rho}(r\log(1/r))^{-2}u^2\ dv
    \leq \int_{B_{r_0}\setminus B_\rho} (\psi u)^2(r\log(1/r))^{-2}\ dv
    \leq 4\int_{B_{r_0}\setminus B_\rho}|\nabla (\psi u)|^2\ dv.
\]
We have chosen $\psi$ so that it has bounded derivatives, so we obtain,
\[ 
     \int_{B_{\rho_1}\setminus B_{\rho}}|\nabla\zeta|^2u^2\ dv
     \leq ct_0^{-2}\int_{B_{r_0}\setminus B_\rho}(u^2+|\nabla u|^2)\ dv.
\]
Since $t_0$ is as large as we like when $\rho$ is chosen small, this completes the proof of the first inequality. The second follows because $|\nabla\zeta|^2$ is large on $A$.
\end{proof}

We will also need an analogous result for the case of half balls.
\begin{lemma} \label{lemma:halfball-cutoff-estimate}
Suppose $B_{r_0}$, a ball in $\mathbb R^n$, is equipped with a metric equivalent to the Euclidean metric. For any $\epsilon$, there are small $\rho$, $\rho_1$ with $\rho<\rho_1 \ll r_0$ and a smooth cut off function $\zeta$, which is $0$ for $x$ in $B_{r_0} \setminus B_{\rho_1}$ and $1$ for $x$ in $B_{\rho}$, such that the following holds. For any smooth function $u$ defined on $B^+_{r_0}$,  
\begin{align*}
       \int_{B^+_{\rho_1}\setminus B^+_{\rho}}|\nabla\zeta|^2u^2\ dv 
       &\leq c\epsilon \left( \int_{B^+_{r_0}\setminus B^+_{\rho}} u^2 \ dv 
                                          + \int_{B^+_{r_0}} |\nabla u|^2 \ dv \right)\\
       \int_{\Gamma_{\rho_1}\setminus\Gamma_\rho}u^2\ da+\int_{B^+_{\rho_1}\setminus B^+_{\rho}}u^2\ dv 
       &\leq c\epsilon \int_{B^+_{r_0}}(u^2+|\nabla u|^2)\ dv
\end{align*}
Here $c$ is a constant depending on $r_0$ and bounds on the eigenvalues of the metric with respect to the Euclidean metric. 
\end{lemma} 
\begin{proof} The first inequality and the following part of the second
\[ 
     \int_{B^+_{\rho_1}\setminus B^+_{\rho}}u^2\ dv 
     \leq c\epsilon\int_{B^+_{r_0}\setminus B^+_{\rho}}(u^2+|\nabla u|^2)\ dv
\]
follow by extending $u$ to $B_{r_0}$ by even reflection and applying the previous lemma.

It remains to prove
\[ \int_{\Gamma_{\rho_1}\setminus\Gamma_\rho}u^2\ da\leq c\epsilon\int_{B^+_{r_0}}(u^2+|\nabla u|^2)\ dv.
\]
We can prove this by using the standard result that on a half ball $B^+_{r_0}$ we have the bound for any $p$ with $2\leq p\leq \frac{2(n-1)}{n-2}$ for $n\geq 3$ and $2\leq p<\infty$ for $n=2$
\[ 
      \left(\int_{\Gamma_{r_0}} u^p\ da\right)^{\frac{2}{p}}
      \leq c\int_{B_{r_0}} (u^2+|\nabla u|^2)\ dv
\]
where $c$ depends on $p$ and $r_0$. We can then fix $p>2$ depending on $n$ and use the H\"older inequality to obtain
\[ 
       \int_{\Gamma_{\rho_1}}u^2\ da\leq c\rho_1^{(n-1)(p-2)/p} \left(\int_{\Gamma_{r_0}}u^p\ da \right)^{\frac{2}{p}}
       \leq c\rho_1^{(n-1)(p-2)/p}\int_{B_{r_0}}(u^2+|\nabla u|^2)\ dv.
\]
Since $\rho_1$ is as small as we wish, this implies the desired bound.
\end{proof}

Finally, we will need the following bound on the $L^2$ norm for functions on a manifold with uniform geometry, in terms of the $L^2$ norm of the function on the boundary and its energy in the interior.

\begin{lemma} \label{lemma:L^2-bound}
For any $W^{1,2}$ function $u$ on an $n$-dimensional Riemannian manifold $(M,g)$ with boundary,
\[
      \int_M u^2 \; dv \leq C\left( \int_{\partial M} u^2 \; da + \int_M |\nabla u|^2 \; dv \right)
\]
where $C$ is a constant depending on $M$. 
\end{lemma}

\begin{proof}
It is sufficient to assume that $u$ is Lipschitz. First consider a rectangular solid $R=(-1,1)^{n-1} \times [0,1) \subset \mathbb{R}^n$ with coordinates $x=(x_1, \ldots, x_n)$. We will use the shorthand notation $x=(x', x_n)$ with $x'=(x_1, \ldots x_{n-1})$. Given a Lipschitz function $u$ on $R$, let
\[
      F(t)=\int_{\{x_n=t\}} u(x',x_n) \; dx'.
\]
Then $\int_R u(x) \; dx = \int_0^1 F(t) \; dt$, and
\[
      F'(t)=\frac{d}{dt} \int_{\{x_n=t\}} u(x',x_n) \; dx' =\int_{\{x_n=t\}} \fder{u}{x_n}(x',x_n) \; dx'.
\]
Therefore, given $t_0 \in [0,1)$,
\begin{equation} \label{equation:slice_estimate}
     F(t_0) - F(0) = \int_0^{t_0} F'(t) \; dt 
      = \int_0^{t_0} \int_{\{x_n=t\}} \fder{u}{x_n}(x',x_n) \; dx' \; dt  
      \leq \int_R |\nabla u | \; dx
\end{equation}
Letting $\overline{u}(t)$ denote the average of $u$ on the slice $\{x_n=t\}$, using (\ref{equation:slice_estimate}) we have
\begin{align*}
     \left(\overline{u}(t_0)\right)^2
     &= \frac{\left(\int_{\{x_n=t_0\}} u(x',x_n) \; dx' \right)^2}{\left(\int_{\{x_n=t_0\}} \; dx' \right)^2} \\
     & \leq 2^{2-2n} \left( \int_{\{x_n=0\}} u(x',x_n) \; dx' + \int_R |\nabla u| \; dx \right)^2 \\
     & \leq 2^{3-2n} \left[ \left(\int_{\{x_n=0\}} u(x',0) \; dx'\right)^2 + \left(\int_R |\nabla u| \; dx \right)^2 \right].
\end{align*}
Using this and the Poincar\'e inequality,
\begin{align*}
     \int_{\{x_n=t\}} & u^2(x',x_n) \; dx' 
     = \int_{\{x_n=t\}} \left( u - \overline{u}(t) \right)^2 \; dx' + \left(\overline{u}(x',t)\right)^2 \\
     & \leq C(n) \int_{\{x_n=t\}} |\nabla u|^2(x',t) \; dx' 
       + 2^{2-n} \left[ \int_{\{x_n=0\}} u^2(x',0) \; dx' + \int_R |\nabla u|^2 \; dx \right]
\end{align*}
Integrating from $t=0$ to $t=1$, we obtain 
\begin{equation} \label{equation:rectangle-bound}
      \int_R u^2 \; dx  \leq C \left( \int_R |\nabla u|^2 \; dx + \int_{\{x_n=0\}} u^2 \; dx' \right)
\end{equation}
where $C$ is a constant depending only on $n$.

We may cover a neighborhood  $S$ of $\partial M$ by a finite number of open sets intersecting $\partial M$, each of which is uniformly bi-Lipschitz equivalent to $R$, with $\partial M$ corresponding to the face with $x_n=0$. Because the metric $g$ is uniformly equivalent to the Euclidean metric on each open set in the cover, each term of the inequality (\ref{equation:rectangle-bound}) only varies within multiplicative bounds determined by that equivalence, and hence 
\begin{equation} \label{equation:strip}
      \int_S u^2 \; dv  \leq C \left( \int_S |\nabla u|^2 \; dv + \int_{\partial M} u^2 \; da \right)
\end{equation}
where $C$ depends on $M$. Now let $\zeta$ be a smooth cut-off function with $\zeta=0$ on $\partial M$ and $\zeta=1$ on $M \setminus S$. By the Poincar\'e inequality,
\begin{align*}
      \int_{M \setminus S} u^2 \; dv & \leq \int_M (\zeta u)^2 \; dv 
       \leq C \int_M |\nabla (\zeta u)|^2 \; dv \\
      & \leq C \int_M \left( |\nabla \zeta|^2 u^2 + \zeta^2 |\nabla u|^2 \right) \; dv \\
      & \leq C \int_S u^2 \; dv + C \int_M |\nabla u|^2 \; dv
\end{align*}
where in the above, the constant $C$ may have increased from one line to the next, but its dependence is always only on $M$. Combing this with (\ref{equation:strip}), we obtain the desired bound
\[
     \int_M u^2 \; dv = \int_{M \setminus S} u^2 \; dv + \int_S u^2 \; dv 
      \leq C\left( \int_{\partial M} u^2 \; da + \int_M |\nabla u|^2 \; dv \right)
\]
where $C$ is a constant depending on $M$. 
\end{proof}

\subsection{Gluing construction and neck estimate} \label{section:gluing-construction}
Let $(M_1,g_1)$ and $(M_2,g_2)$ be compact $n$-dimensional Riemannian manifolds with nonempty boundary. We glue $M_1$ and $M_2$ together along their boundaries as follows. Let $p_1 \in \partial M_1$ and $p_2 \in \partial M_2$. Choose $r_0>0$ such that the metrics $g_1$, $g_2$ are uniformly equivalent to the Euclidean metric in balls of radius $r_0$. Given $\rho >0$ sufficiently small, and $\rho_1=\rho_1(\rho)$ with $0 < \rho < \rho_1 \ll r_0$ as in the proof of Lemma \ref{logcutoff}, choose a smooth metric $g_{i,\rho}$ on $M_i$ such that $g_{i,\rho}$ is flat on the geodesics ball $\mathcal{B}^i_{\rho_1}(p_i)$ of radius $\rho_1$ in $M_i$ centered at $p_i$ and equal to $g_i$ on $M_i \setminus \mathcal{B}^i_{2\rho_1}(p_i)$ for $i=1,\,2$.

For $n \geq 3$, consider a catenoid in $\mathbb{R}^n$; that is, a complete minimal hypersurface of revolution that is not a hyperplane. A catenoid is parametrized by an embedding
\[
       F: I \times \mathbb{S}^{n-2} \rightarrow \mathbb{R}^n
\]
with $F(t, \omega)=(\phi(t)\omega, t)$, where $\phi: I \rightarrow \mathbb{R}$ is a solution, defined on a maximal interval $I=(-a(n),a(n))$, of an ODE corresponding the the minimal surface equation. There exists $l:=l(n) < a(n)$, such that the portion of the catenoid corresponding to $-l \leq t \leq l$ is volume minimizing \cite[Corollary 3]{S}. Given $\rho>0$, consider a rescaled portion of the catenoid given by
\[
      \tilde{F}: [-l,l] \times \mathbb{S}^{n-2} \rightarrow \mathbb{R}^n
\]
with
\[
      \tilde{F}(t, \omega)=\frac{1}{R} (\phi(t)\omega,t)
\]
where $R=\rho/\phi(l)$. Then consider the solid catenoidal tube 
\[
     T_\rho:=\tilde{F}([-l,l] \times \overline{\mathbb{B}}^{n-2}).
\]
Note that the ends of $T_\rho$, corresponding to $t=\pm l$, are Euclidean balls of radius $\rho$, and the catenoid portion $\tilde{F}([-l,l] \times \mathbb{S}^{n-2})$ of the boundary of $T_\rho$ is a volume minimizing hypersurface. For $n=2$, let $T_\rho$ be a Euclidean square of side length $2\rho$. In this case note that the boundary portion of $T_\rho$ consisting of two aligned parallel line segments of length $2\rho$ that are a distance $2\rho$ apart is length minimizing with respect to its boundary points.

We now let $M_\rho$ be the Lipschitz Riemannian manifold obtained by gluing $(M_1,g_{1,\rho})$ and $(M_2,g_{2,\rho})$ together along their boundaries using the tube $T_\rho$. Specifically, $M_\rho$ is obtained by identifying one end of $T_\rho$ with $\partial M_1 \cap \mathcal{B}^1_{\rho}(p_1)$, and the other end of $T_\rho$ with $\partial M_2 \cap \mathcal{B}^2_{\rho}(p_2)$. Let
\[
      N_\rho:= \mathcal{B}^1_{\rho_1}(p_1) \cup T_\rho \cup \mathcal{B}^2_{\rho_1}(p_2),
\]
with the identifications  as above. $N_\rho$ is a Euclidean domain with piecewise smooth boundary.

An important ingredient in the proof of Theorem \ref{theorem:gluing-boundary} is that for a sequence of eigenfunctions, the $L^2$ norm on the boundary $\partial M_\rho$ doesn't concentrate on the boundary of the tube $\partial M_\rho \cap \partial T_\rho$ as $\rho \rightarrow 0$. In order to prove this, we will need the following two lemmas. The first lemma uses in a key way the geometry of the neck region $N_\rho$.

\begin{lemma} \label{lemma:volume-comparison}
Let $f: N_\rho \rightarrow \mathbb{R}$ be a smooth function with $f \geq 0$ on $N_\rho$ and $f=0$ on $\partial \mathcal{B}^i_{\rho_1}(p_i) \setminus \partial M_i$ for $i=1, \, 2$. Then
\[
     \mbox{Vol}(\{ x \in \partial N_\rho : \, f(x) > t \}) \leq \mbox{Vol}(\{ x \in N_\rho : \, f(x) = t \}).
\]
\end{lemma}

\begin{proof}
Let $\Omega_t=\{ x \in N_\rho : f(x) > t \}$ for $t>0$. Then
$\partial \Omega_t = ( \partial \Omega_t \cap \partial N_\rho) \cup (\partial \Omega_t \setminus \partial N_\rho)$
with
\begin{align*}
    \partial \Omega_t \cap \partial N_\rho & = \{ x \in \partial N_\rho : f(x) \geq t \} \\
    \partial \Omega_t \setminus \partial N_\rho & = \{ x \in N_\rho : f(x) = t \}.
\end{align*}

We first consider the portion of $\partial \Omega_t \cap \partial N_\rho$ that lies on $\partial M_i$, $i=1, \, 2$.
Recall that $\mathcal{B}^i_{\rho_1}(p_i)$ is a Euclidean half ball, and observe that the orthogonal projection maps $P_i: \mathcal{B}^i_{\rho_1}(p_i) \rightarrow \mathcal{B}^i_{\rho_1}(p_i) \cap \partial M_i$ for $i=1, \, 2$ reduce volumes of hypersurfaces. Also, since $f=0$ on $\partial \mathcal{B}^i_{\rho_1}(p_i) \setminus \partial M_i$, it follows that $\overline{\Omega}_t \cap (\partial \mathcal{B}^i_{\rho_1}(p_i) \setminus \partial M_i) = \emptyset$. This, together with the fact that $(\mathcal{B}^i_{\rho_1}(p_i) \setminus \mathcal{B}^i_{\rho}(p_i)) \cap \partial M_i$ is connected and flat, implies that the line orthogonal to $\partial M_i$ through any point $x \in \partial \Omega_t \cap \partial M_i$ with $f(x)>0$, must intersect $\partial \Omega_t$ at a second point. Therefore,
\begin{align} \label{equation:end-volume}
     \mbox{Vol}(\partial \Omega_t \cap \partial M_i) 
     & \leq \mbox{Vol} \left( 
     P_i \left(( \partial \Omega_t \setminus \partial N_\rho) \cap (\mathcal{B}^i_{\rho_1}(p_i) \setminus C)\right)
                                  \right)  \no \\
     & \leq \mbox{Vol}
     \left( ( \partial \Omega_t \setminus \partial N_\rho) \cap (\mathcal{B}^i_{\rho_1}(p_i) \setminus C) \right)
\end{align}
where $C$ is the solid cylinder of radius $\rho$ with axis through $p_i$ orthogonal to $\partial M_i$.

We next consider the remaining portion of $\partial \Omega_t \cap \partial N_\rho$, which lies on $\partial T_\rho \cap \partial N_\rho$. Let $\Omega_t'= \Omega_t \cap T_\rho$. Then clearly, $\mbox{Vol}(\partial \Omega_t' \setminus \partial N_\rho) \leq \mbox{Vol}((\partial \Omega_t \setminus \partial N_\rho) \cap C)$, since $\partial \Omega_t' \setminus \partial N_\rho \subset C$ and the portions of $\partial \Omega_t' \setminus \partial N_\rho$ that differ from $\partial \Omega_t \setminus \partial N_\rho \cap C$ consist of subsets of $\mathcal{B}^i_{\rho}(p_i) \cap \partial M_i$ that are contained in the orthogonal projection of $(\partial \Omega_t \setminus \partial N_\rho) \cap C$ onto the flat ball $\mathcal{B}^i_{\rho}(p_i) \cap \partial M_i$. Furthermore, $\partial \Omega_t' \cap (\partial T_\rho \cap \partial N_\rho) = \partial \Omega_t \cap (\partial T_\rho \cap \partial N_\rho)$. 
Since $\partial T_\rho \cap \partial N_\rho$ is volume minimizing and $\partial \Omega_t \cap (\partial T_\rho \cap \partial N_\rho)$ and $\partial \Omega_t' \setminus \partial N_\rho$ have the same boundary,
\begin{equation} \label{equation:tube-volume}
       \mbox{Vol}(\partial \Omega_t \cap (\partial T_\rho \cap \partial N_\rho))
       \leq \mbox{Vol}(\partial \Omega_t' \setminus \partial N_\rho)
       \leq \mbox{Vol}( (\partial \Omega_t \setminus \partial N_\rho) \cap C).
\end{equation}
Combining (\ref{equation:end-volume}) and (\ref{equation:tube-volume}), we obtain the desired volume comparison,
\[
      \mbox{Vol}(\{ x \in \partial N_\rho : \, f(x) > t \}) = \mbox{Vol} (\partial \Omega_t \cap \partial N_\rho)
      \leq \mbox{Vol} (\partial \Omega_t \setminus \partial N_\rho)
      = \mbox{Vol}(\{ x \in N_\rho : \, f(x) = t \}).     
\]
\end{proof}

As a consequence of the previous lemma, we have the following.

\begin{lemma}  \label{lemma:poincare}
Let $w$ be a smooth function on $N_\rho$ with $w=0$ on $\partial \mathcal{B}^i_{\rho_1}(p_i) \setminus \partial M_i$ for $i=1, \, 2$. Then
\begin{align*}
            \int_{N_\rho} w^2 \; dv  &\leq C(n) \rho \int_{N_\rho} |\nabla w|^2 \; dv \\
       \int_{\partial N_\rho} w^2 \; da & \leq C(n) \sqrt{\rho} \int_{N_\rho} |\nabla w|^2 \; dv.
\end{align*}
\end{lemma}

\begin{proof}
In what follows, $C(n)$ may increase from one line to the next, but its dependence will always be only on $n$.
First consider the Euclidean domain $N=N_\rho$ with $\rho=1/2$. Let $f: N \rightarrow \mathbb{R}$ be a smooth function with $f \geq 0$ on $N$ and $f=0$ on $\partial \mathcal{B}^i_{\rho_1}(p_i) \setminus \partial M_i$ for $i=1, \, 2$. Let $\Omega_t=\{ x \in N : f(x) > t \}$. By the isoperimetric inequality and Lemma \ref{lemma:volume-comparison},
\begin{align*}
      \mbox{Vol}(\{ x \in N : \, f(x) > t \}) & = \mbox{Vol}(\Omega_t) 
      \leq C(n) \mbox{Vol}(\partial \Omega)^{\frac{n}{n-1}} \\
      &= C(n) \left( \mbox{Vol}(\partial \Omega_t \cap \partial N) 
                    + \mbox{Vol}(\partial \Omega_t \setminus \partial N)\right)^{\frac{n}{n-1}} \\
      & \leq C(n) 2^{\frac{n}{n-1}} \mbox{Vol}(\partial \Omega_t \setminus \partial N)^{\frac{n}{n-1}} \\
      &= C(n) 2^{\frac{n}{n-1}} \mbox{Vol}(\{x \in N : \, f(x)=t \})^{\frac{n}{n-1}}.
\end{align*}
This implies the Sobolev inequality (see for example \cite[page 90]{SY}),
\[
      \left(\int_{N} f^{\frac{n}{n-1}} \right)^{\frac{n-1}{n}} \; dv  \leq C(n) \int_{N} |\nabla f| \; dv.
\]
Then using H\"older's inequality we obtain
\[
     \int_{N} f \; dv \leq C(n) \mbox{Vol}(N)^{\frac{1}{n}} \int_{N} |\nabla f| \; dv.
\]
Applying this to the function $f=w^2$, we obtain
\begin{align*}
     \int_{N} w^2 \; dv & \leq C(n) \int_{N} |\nabla w^2| \; dv 
     = C(n) \int_{N} 2|w||\nabla w| \; dv \\
     & \leq 2 C(n) \left( \int_{N} w^2 \; dv \right)^{\frac{1}{2}} \left( \int_{N} |\nabla w|^2 \; dv \right)^{\frac{1}{2}} \\
     & \leq \frac{1}{2} \int_{N} w^2 \; dv + 2 C(n)^2 \int_{N} |\nabla w|^2 \; dv.
\end{align*}      
Therefore,
\begin{equation} \label{equation:interior}
      \int_{N} w^2 \; dv \leq C(n) \int_{N} |\nabla w|^2 \; dv.
\end{equation}
Scaling the domain by a factor of $2\rho$, we obtain the desired estimate
\[
     \int_{N_\rho} w^2 \; dv
     \leq C(n) \rho \int_{N_\rho} |\nabla w|^2 \; dv.
\]
Similarly, by Lemma \ref{lemma:volume-comparison} and the co-area formula, we have
\begin{align*}
     \int_{\partial N} f \; da 
     &= \int_0^\infty \mbox{Vol}(\{ x \in \partial N : \, f > t \}) \; dt \\
     & \leq \int_0^\infty \mbox{Vol}(\{ x \in N : \, f = t \}) \; dt \\
     & = \int_{N} |\nabla f| \; dv.
\end{align*}
Applying this to the function $f=w^2$, we obtain
\[ 
     \int_{\partial N} w^2 \; da \leq \int_{N} | \nabla w^2| \; dv 
      \leq \frac{1}{2} \int_{N} w^2 \; dv + 2 \int_{N} |\nabla w|^2 \; dv 
      \leq C(n) \int_{N} |\nabla w|^2 \; dv,
\] 
where we have used (\ref{equation:interior}) in the last inequality.
Finally, scaling the domain by a factor of $2\rho$, we have
\[
     \int_{\partial N_\rho} w^2 \; da
     \leq C(n) \sqrt{\rho} \int_{N_\rho} |\nabla w|^2 \; dv.
\]
\end{proof}

The following estimate on the neck region will be important in the proof of Theorem \ref{theorem:gluing-boundary}.

\begin{proposition} \label{proposition:neck-estimate}
For any $\epsilon>0$, there is a small $\rho \ll r_0$, such that for any smooth function $u$ on $M_\rho$,
\begin{align*}
       \int_{\partial T_\rho \cap \partial M_\rho} u^2 \; dv 
       &\leq c\epsilon \left( \int_{N_\rho} |\nabla u|^2 \; dv 
     + \int_{\cup_{i=1}^2 \mathcal{B}_{r_0}^i(p_i ) \setminus \mathcal{B}_{\rho}^i(p_i )} u^2 \; dv \right) \\
       \int_{\cup_{i=1}^2 \mathcal{B}_{\rho}^i(p_i ) \cup T_\rho} u^2 \; dv
       &\leq c\epsilon \left( \int_{N_\rho} |\nabla u|^2 \; dv 
     + \int_{\cup_{i=1}^2 \mathcal{B}_{r_0}^i(p_i ) \setminus \mathcal{B}_{\rho}^i(p_i )} u^2 \; dv \right).
\end{align*}
\end{proposition}

\begin{proof}
Let $w=\zeta u$ where $\zeta$ is the smooth cut-off function from Lemma \ref{lemma:halfball-cutoff-estimate} which is 1 on $\mathcal{B}_{\rho}^1(p_1) \cup \mathcal{B}_{\rho}^2(p_2) \cup T_\rho$ and 0 on $M_\rho \setminus (\mathcal{B}_{\rho_1}^1(p_1) \cup \mathcal{B}_{\rho_1}^2(p_2))$.  By Lemma \ref{lemma:poincare},
\begin{align*}
     \int_{\partial T_\rho \cap \partial M_\rho} u^2 \; da & \leq \int_{\partial N_\rho} (\zeta u)^2 \; da \\
     & \leq \epsilon \int_{N_\rho}  |\nabla(\zeta u)|^2 \; dv \\
     & \leq  2 \epsilon \int_{N_\rho} |\nabla u|^2 \; dv 
       + 2 \epsilon \int_{\cup_{i=1}^2 \mathcal{B}_{\rho_1}^i(p_i ) \setminus \mathcal{B}_{\rho}^i(p_i )} 
                      |\nabla \zeta|^2 u^2  \; dv \\
     & \leq c\epsilon \left( \int_{N_\rho} |\nabla u|^2 \; dv 
     + \int_{\cup_{i=1}^2 \mathcal{B}_{r_0}^i(p_i ) \setminus \mathcal{B}_{\rho}^i(p_i )} u^2 \; dv \right)                
\end{align*}
where the last inequality follows from Lemma \ref{lemma:halfball-cutoff-estimate}. The proof of the second inequality is analogous.
\end{proof}

\subsection{Proof of continuity of Steklov eigenvalues under certain degenerations}
In this section we give the proof of Theorem \ref{theorem:gluing-boundary} and related results. Using Proposition \ref{proposition:neck-estimate}, which implies that for a sequence of eigenfunctions the $L^2$ norm on the boundary of $M_\rho$ doesn't concentrate on the neck as $\rho \rightarrow 0$, the proof of the gluing theorem is similar to the proof of Proposition 4.1 of \cite{FS5}.
\begin{proof}[Proof of Theorem \ref{theorem:gluing-boundary}]
We will prove the result for $s=2$, although the same argument works for gluing any number $s \geq 2$ of manifolds. Let $M_\rho$ be the Lipschitz Riemannian manifold defined in section \ref{section:gluing-construction}. First we locally smooth the corners of $M_\rho$. Specifically, there exists a bi-Lipschitz map $F: M_\rho \rightarrow \tilde{M}_\rho$, where $\tilde{M}_\rho$ is a smooth Riemannian manifold, such that $F$ and $F^{-1}$ have bounded Lipschitz constant independent of $\rho$. Note that the estimates of Lemma \ref{lemma:poincare} and Proposition \ref{proposition:neck-estimate} carry over to $\tilde{M}_\rho$ under the bi-Lipschitz equivalence, since $F$ and $F^{-1}$ have bounded Lipschitz constant independent of $\rho$. For notational simplicity we will write $M_\rho$, instead of $\tilde{M}_\rho$, for the smoothed manifold. 

Let $0=\sigma_0(M_\rho) \leq \sigma_1(M_\rho) \leq \sigma_2(M_\rho) \leq \cdots$ be the Steklov eigenvalues of $M_\rho$ and let $u_\rho^{(0)}, u_\rho^{(1)},  u_\rho^{(2)}, \ldots$ be orthonormal eigenfunctions; i.e. $\|u_\rho^{(k)}\|_{L^2(\partial M_\rho)}=1$,
\[
          \int_{\partial M_\rho} u_\rho^{(k)} \, u_\rho^{(l)} \; da =0 \qquad \mbox{ for } k \neq l
\]
and
\begin{align*}
       \begin{cases}
       \Delta u^{(k)}_\rho =0 & \mbox{ on } M_\rho \\
       \fder{u^{(k)}_\rho}{\nu}=\sigma_k(M_\rho) \, u^{(k)}_\rho & \mbox{ on } \partial M_\rho.
       \end{cases}
\end{align*}

We first show that each $\sigma_k(M_\rho)$ is bounded from above by a constant $\Lambda_k$ independent of $\rho$ for $\rho$ small. To see this we use the variational characterization of $\sigma_k$
\[ 
    \sigma_k(M_\rho)
  =\inf_E \sup \{ \frac{\int_{M_\rho}|\nabla \hat{f}|^2}{\int_{\partial M_\rho}f^2}:\ f \in E, f \neq 0\}
\]
where the infimum is taken over all $(k+1)$-dimensional subspaces $E$ of $L^2(\partial M_\rho)$, and $\hat{f}$ denotes the harmonic extension of $f$ to $M_\rho$. Thus to get an upper bound we need only exhibit $k+1$ linearly independent functions having bounded Rayleigh quotient. We can do this by choosing $k+1$ {\it fixed} such functions which are supported away from the neck region $N_\rho$ and so are valid test functions for any small $\rho$.

Since $u_\rho^{(k)}$ is a Steklov eigenfunction of $M_\rho$ with eigenvalue $\sigma_k(M_\rho)$, 
\[ 
       \int_{M_\rho} |\nabla u_\rho^{(k)}|^2 \; dv = \sigma_k(M_\rho) \int_{\partial M_\rho} (u_\rho^{(k)})^2 \; da
       = \sigma_k(M_\rho)  \leq \Lambda_k.
\] 
By Lemma \ref{lemma:L^2-bound},
\begin{equation} \label{equation:L^2-bound}
          \int_K (u_\rho^{(k)})^2 \; dv 
          \leq C \left( \int_K |\nabla u_\rho^{(k)}|^2 \; dv + \int_{\partial M_\rho \cap K} (u_\rho^{(k)})^2 \; da \right)
          \leq C(\Lambda_k+1)
\end{equation}
for any compact subset $K$ of $(M_1 \setminus \{p_1\}) \sqcup (M_2 \setminus \{p_2\})$ for all sufficiently small $\rho$, where $C=C(M_1, M_2)$. This together with Proposition \ref{proposition:neck-estimate} implies that we have a uniform bound (independent of $\rho$) on the $L^2$ norm of $u_\rho^{(k)}$ on $M_\rho$. Hence, there exists $C>0$ independent of $\rho$ such that for all sufficiently small $\rho$,
\begin{equation} \label{equation:W1,2-bound}
     \|u^{(k)}_{\rho}\|_{W^{1,2}(M_{\rho})} \leq C(k, M_1, M_2).
\end{equation}

Elliptic boundary estimates ({\cite[Theorem 6.30]{GT}}) give uniform bounds 
\[
       \|u^{(k)}_\rho \|_{C^{2,\alpha}(K)} \leq C \| u_\rho^{(k)}\|_{C^0(K)}
\]
for any compact subset $K$ of $(M_1 \setminus \{p_1\}) \sqcup (M_2 \setminus \{p_2\})$ for all sufficiently small $\rho$, where $C=C(k,\alpha,\Lambda_k, M_1, M_2)$. 
By Sobolev embedding and interpolation inequalities (\cite[Theorem 5.2]{AF}, \cite[(7.10)]{GT}), \
\[
     \|u_\rho^{(k)}\|_{C^0(K)} \leq C \left( \varepsilon\|u_\rho^{(k)}\|_{C^{2}(K)} 
                                                   + \varepsilon^{-\mu} \|u_\rho^{(k)}\|_{L^2(K)} \right)
\]
where $\epsilon >0$ can be taken arbitrarily small, $\mu>0$ depends on $n$, and $C$ depends on $M_1$, $M_2$. Hence $\|u^{(k)}_\rho \|_{C^{2,\alpha}(K)} \leq C$ with $C$ independent of $\rho$.
By the Arzela-Ascoli theorem and a diagonal argument, there exists a sequence $\rho_i \rightarrow 0$ such that for all $k$, $u^{(k)}_{\rho_i}$ converges in $C^2(K)$ on compact subsets $K \subset (M_1 \setminus \{p_1\}) \sqcup (M_2 \setminus \{p_2\})$ to a harmonic function $u^{(k)}$ on $(M_1,g_1) \sqcup (M_2,g_2)$, satisfying 
\[
       \fder{u^{(k)}}{\nu}=\sigma_k u^{(k)} 
       \quad \mbox{ on } \quad (\partial M_1 \setminus \{p_1\}) \sqcup (\partial M_2 \setminus \{p_2\}),
\]
with $\sigma_k =\lim_{i \rightarrow \infty} \sigma_k(M_{\rho_i})$. 

We now show that $u^{(k)}$ extends to a Steklov eigenfunction on $M_1 \sqcup M_2$. Consider the logarithmic cut-off function $\varphi_\rho$ that is equal to 1 on $M_j \setminus \mathcal{B}_{\sqrt{\rho}}^j(p_j)$, equal to zero on $\mathcal{B}_\rho^j(p_j)$ and is given by
\begin{equation} \label{equation:cut-off function}
 	\varphi_{\rho}(r) = \frac{\log r - \log \rho }{- \log \sqrt{\rho}} \quad \mbox{ for } \rho \le  r \le \sqrt{\rho} 
 \end{equation}
where $r$ is the radial distance from $p_j$, for $j=1,\,2$.
Then,
\begin{align}
     \int_{M_1 \sqcup M_2}  |\nabla \varphi_\rho|^2 \; dv
     &\leq \sum_{j=1}^2 \int_{\mathcal{B}^j_{\sqrt{\rho}}(p_j) \setminus \mathcal{B}^j_{\rho}(p_j)}
              |\nabla \varphi_\rho|^2  \; dv  \no \\
     &= \frac{C(n)}{(\log \sqrt{\rho})^2} 
                 \int_{\rho}^{\sqrt{\rho}} r^{n-3}  \; dr   \no \\
     &= C(n) \epsilon_n(\rho) \rightarrow 0 \mbox{ as } \rho \rightarrow 0 \label{equation:cut-off}
\end{align} 
where $\epsilon_2(\rho)=1/\log(1/\sqrt{\rho})$ and 
$\epsilon_n(\rho)=\rho^{\frac{n-2}{2}}(1-\rho^{\frac{n-2}{2}})/[(n-2)(\log\sqrt{\rho})^2]$ for $n \geq 3$.
Let $\psi \in W^{1,2} \cap L^\infty (M_1 \sqcup M_2)$ and let $\psi_\delta=\varphi_\delta \psi$. Since $u^{(k)}$ is a harmonic function on $M_1 \sqcup M_2$, satisfying 
\[
       \fder{u^{(k)}}{\nu}=\sigma_k u^{(k)} 
       \quad \mbox{ on } \quad \partial (M_1 \sqcup M_1) \setminus \{p_1,\, p_2\},
\]
and $\psi_\rho$ vanishes near $p_1, \, p_2$, we have
\begin{equation} \label{equation:weak-steklov}
    \int_{M_1 \sqcup M_2} \nabla u^{(k)} \nabla \psi_\rho \; dv
    = \sigma_k \int_{\partial (M_1 \sqcup M_2)} u^{(k)} \psi_\rho \; da.
\end{equation}
By (\ref{equation:cut-off}) and H\"older's inequality,
\[
     \int_{M_1 \sqcup M_2} \psi \nabla u^{(k)} \nabla \varphi_\rho \rightarrow 0
     \qquad \mbox{ as } \rho \rightarrow 0.
\]
Since $|\psi_\rho| \leq |\psi| \in L^\infty$ and $\psi_\rho \rightarrow \psi$ a.e., by the dominated convergence theorem, taking the limit of (\ref{equation:weak-steklov}) as $\rho \rightarrow 0$, we obtain
\[
    \int_{M_1 \sqcup M_2} \nabla u^{(k)} \nabla \psi \; dv
    = \sigma_k \int_{\partial (M_1 \sqcup M_2)} u^{(k)} \psi \; da.      
\]
Therefore, $u^{(k)}$ extends to a Steklov eigenfunction with eigenvalue $\sigma_k$ on $M_1 \sqcup M_2$.

Now observe that $\{ u^{(k)}\}_{k=1}^\infty$ are $L^2$-orthonormal on $\partial (M_1 \sqcup M_2)$. Since $\{u^{(k)}_{\rho_i}\}_{k=1}^\infty$ are $L^2$-orthonormal on $\partial M_{\rho_i}$, 
\begin{align} \label{equation:orthonormal}
      \delta_{kl} & = \lim_{i \rightarrow \infty} \int_{\partial M_{\rho_i}} u^{(k)}_{\rho_i}u^{(l)}_{\rho_i} \; da \no \\
      & =\lim_{i \rightarrow \infty} \left[\sum_{j=1}^2 
      \int_{\partial M_j \setminus \mathcal{B}^j_{\rho_i}(p_j) } 
                    u^{(k)}_{\rho_i}u^{(l)}_{\rho_i} \; da
          + \int_{\partial M_{\rho_i} \cap \partial T_{\rho_i}} u^{(k)}_{\rho_i}u^{(l)}_{\rho_i} \; da\right]  \no \\
      & = \sum_{j=1}^2 \int_{\partial M_j} \lim_{i \rightarrow \infty}
   \chi_{\partial M_j \setminus \mathcal{B}^j_{\rho_i}(p_j) } 
                   u^{(k)}_{\rho_i}u^{(l)}_{\rho_i} \; da \no \\
      & = \int_{\partial M_1 \sqcup \partial M_2} u^{(k)} u^{(l)} \; da
\end{align}
where the third equality follows from the bounded convergence theorem, and since by (\ref{equation:W1,2-bound}) and Proposition \ref{proposition:neck-estimate} we have $\lim_{i \rightarrow \infty} \|u^{(k)}_{\rho_i} \|_{L^2(\partial M_{\rho_i} \cap \partial T_{\rho_i})}=0$.
Here the domains are understood to be the corresponding domains under the bi-Lipschitz map $F$.

Finally, we show that $u^{(k)}$ is a $k$-th eigenfunction of $M_1 \sqcup M_2$; i.e. $\sigma_k=\sigma_k(M_1 \sqcup M_2)$. We prove this by induction on $k$. First, since $\sigma_0(M_\rho)=0$, we have that $\sigma_0= \lim_{\rho \rightarrow 0} \sigma_0(M_\rho)=0$, and so $\sigma_0=\sigma_0(M_1 \sqcup M_2)$. Now suppose $\sigma_l=\sigma_l(M_1 \sqcup M_2)$ for $l=1, \ldots, k-1$, where $k \geq 1$. We will show that $\sigma_k=\sigma_k(M_1 \sqcup M_2)$. It follows from (\ref{equation:orthonormal}) that $\sigma_k \geq \sigma_k(M_1 \sqcup M_2)$. It remains to show that $\sigma_k \leq \sigma_k(M_1 \sqcup M_2)$.
 
Let $w$ be a $k$-th eigenfunction of $M_1 \sqcup M_2$ with $\|w\|_{L^2(\partial (M_1 \sqcup M_2))}=1$, and let
\[
    w_\rho =\varphi_\rho w 
       - \sum_{l=1}^{k-1} \left( \int_{\partial M_\rho} (\varphi_\rho w) u_\rho^{(l)} \; da \right) u_\rho^{(l)}
\]
where $\varphi_\rho$ is the logarithmic cut-off function defined by (\ref{equation:cut-off function}). We may then use $w_\delta$ as a test function in the variational characterization of $\sigma_k(M_\rho)$. First note that
\[
    \int_{\partial M_\rho} w_\rho^2 \; da = \int_{\partial M_\rho} (\varphi_\rho w)^2 \; da  - \sum_{l=1}^{k-1} 
    \left( \int_{\partial M_\rho} (\varphi_\rho w) u_\rho^{(l)} \; da \right)^2.
\]
But 
\begin{equation} \label{equation:limit-orthogonal}
     \lim_{i \rightarrow \infty} \int_{\partial M_{\rho_i}} (\varphi_{\rho_i} w) u_{\rho_i}^{(l)} \; da
    =\int_{\partial(M_1 \sqcup M_2)} w u^{(l)} \; da =0,
\end{equation}
using an argument as in (\ref{equation:orthonormal}), where the last equality follows since $w$ is a $k$-th eigenfunction of $M_1 \sqcup M_2$. Therefore,
\[
     \lim_{i \rightarrow \infty} \int_{\partial M_{\rho_i}} w_{\rho_i}^2 \; da 
     = \lim_{i \rightarrow \infty} \int_{\partial M_{\rho_i}} (\varphi_{\rho_i} w)^2 \; da
     = \int_{\partial(M_1 \sqcup M_2)} w^2 \; da.
\]
On the other hand, 
\[
     \int_{M_\rho}  |\nabla (\varphi_\rho w) |^2 \; dv 
      \leq \int_{M_\rho} \varphi_\rho^2 |\nabla w|^2 \; dv + C \int_{M_\rho} |\nabla \varphi_\rho|^2 \; dv
      \stackrel{\rho \rightarrow 0}{\longrightarrow} \int_{M_1 \sqcup M_2} |\nabla w|^2 \; dv
\]
using (\ref{equation:cut-off}),
where the constant $C$ depends on a pointwise upper bound on $w$ and $|\nabla w|$.
Using this together with (\ref{equation:W1,2-bound}) and (\ref{equation:limit-orthogonal}) we deduce that
\[
     \lim_{i \rightarrow \infty} \int_{M_{\rho_i}} |\nabla w_\rho |^2 \; dv 
     \leq \int_{M_1 \sqcup M_2} |\nabla w|^2 \; dv. 
\]
Combining these estimates, we have 
\[
    \sigma_k = \lim_{i \rightarrow \infty} \sigma_k(M_{\rho_i}) 
     \leq \lim_{i \rightarrow \infty} 
     \frac{\int_{M_{\rho_i}} |\nabla w_{\rho_i} |^2 \; dv}{\int_{\partial M_{\rho_i}} w_{\rho_i}^2 \; da} 
    \leq \frac{\int_{M_1 \sqcup M_2} |\nabla w|^2 \; dv}{\int_{\partial (M_1 \sqcup M_2)} w^2 \; da} 
    = \sigma_k(M_1 \sqcup M_2).
\]
Therefore,
\[
     \lim_{i \rightarrow \infty} \sigma_k(M_{\rho_i})= \sigma_k(M_1 \sqcup M_2).
\]
Clearly, $\lim_{\rho \rightarrow 0} |\partial M_\rho| = |\partial (M_1 \sqcup M_2)|$. 
\end{proof}

We remark that the same argument can be used to glue a single manifold to itself along its boundary.

\begin{theorem} \label{theorem:gluing-single-boundary}
Let $M$ be an $n$-dimensional Riemannian manifold with nonempty boundary. Given any $\epsilon >0$ there exists a manifold $M_\epsilon$ obtained by gluing $M$ to itself along its boundary, along neighborhoods of distinct boundary points, such that 
\[
     \lim_{\epsilon \rightarrow 0} |\partial M_\epsilon | = |\partial M| 
     \; \mbox{ and } \; \lim_{ \epsilon \rightarrow 0} \sigma_k(M_\epsilon) = \sigma_k(M)  
\]
for $k=0, \, 1,\, 2, \ldots$.
\end{theorem} 

Using similar methods, we obtain an analogous result showing that the first $k$ Steklov eigenvalues are continuous under certain degenerations along the interior rather than the boundary.

\begin{customthm}{\ref{theorem:gluing-interior}}
Let $M_1, \ldots, M_s$ be compact n-dimensional Riemannian manifolds with nonempty boundary. Given $\epsilon >0$ there exists a Riemannian manifold $M_\epsilon$, obtained by appropriately gluing $M_1, \ldots, M_s$ together along there interiors, such that 
$\partial M_\epsilon=\partial (M_1 \sqcup \ldots \sqcup M_s)$ and
\[
       \lim_{ \epsilon \rightarrow 0} \sigma_k (M_\epsilon) =  \sigma_k(M_1 \sqcup \cdots \sqcup M_s)
\]
for $k=0, \, 1,\, 2, \ldots$.
\end{customthm}

The proof is analogous to the proof of Theorem \ref{theorem:gluing-boundary}, yet significantly easier, since 
the delicate neck estimates of sections \ref{section:preliminaries} and \ref{section:gluing-construction} are not needed in this case. 

\begin{proof}
We will prove the result for $s=2$, although the same argument works for gluing any number $s \geq 2$ of manifolds. Let $(M_1,g_1)$ and $(M_2,g_2)$ be compact $n$-dimensional Riemannian manifolds with nonempty boundary, and let $p_1 \in \mbox{Int}\, M_1$ and $p_2 \in \mbox{Int}\, M_2$. Given $\rho >0$ sufficiently small, choose a smooth metric $g_{i,\rho}$ on $M_i$ such that $g_{i,\rho}$ is flat on the geodesics ball $\mathcal{B}^i_{\rho}(p_i)$ of radius $\rho$ in $M_i$ centered at $p_i$ and equal to $g_i$ on $M_i \setminus \mathcal{B}^i_{2\rho}(p_i)$ for $i=1,\,2$. Let $T_\rho=\mathbb{S}^{n-1}(\rho) \times \mathbb{R}$ with the standard product metric, and let $M_\rho$ be the Lipschitz Riemannian manifold obtained by gluing $M_1 \setminus \mathcal{B}_\rho^1(p_1)$ to $M_2 \setminus \mathcal{B}_\rho^2(p_2)$ using $T_\rho$, by identifying one end of $T_\rho$ with $\partial \mathcal{B}^1_{\rho}(p_1)$ and the other end of $T_\rho$ with $\partial  \mathcal{B}^2_{\rho}(p_2)$. We may then locally smooth out the corners of $M_\rho$ to obtain a smooth Riemannian manifold, which we will continue to denote by $M_\rho$.  Note that $\partial M_\rho=\partial M_1 \sqcup \partial M_2$.

Let $0=\sigma_0(M_\rho) \leq \sigma_1(M_\rho) \leq \sigma_2(M_\rho) \leq \cdots$ be the Steklov eigenvalues of $M_\rho$ and let $u_\rho^{(0)}, u_\rho^{(1)},  u_\rho^{(2)}, \ldots$ be a complete sequence of eigenfunctions that are $L^2$-orthonormal on $\partial M_\rho$, such that $u_\rho^{(k)}$ is an eigenfunction of $\sigma_k(M_\rho)$. As in the proof of Theorem \ref{theorem:gluing-interior}, $\int_{M_\rho} |\nabla u_\rho^{(k)}|^2 \; dv= \sigma_k(M_\rho) \leq \Lambda_k$, with $\Lambda_k$ independent of $\rho$. Elliptic boundary estimates, interpolation inequalities and Lemma \ref{lemma:L^2-bound}, give uniform bounds
\[
       \|u^{(k)}_\rho \|_{C^{2,\alpha}(K)} \leq C \| u_\rho^{(k)}\|_{L^2(K)} \leq C(k,\alpha,\Lambda_k, M_1, M_2)
\]
for any compact subset $K$ of $(M_1 \setminus \{p_1\}) \sqcup (M_2 \setminus \{p_2\})$. By the Arzela-Ascoli theorem, there exists a sequence $\rho_i \rightarrow 0$ such that for all $k$, $u^{(k)}_{\rho_i}$ converges in $C^2(K)$ on compact subsets $K \subset (M_1 \setminus \{p_1\}) \sqcup (M_2 \setminus \{p_2\})$ to a harmonic function $u^{(k)}$ on $(M_1 \setminus \{p_1\},g_1) \sqcup (M_2 \setminus \{p_2\},g_2)$, satisfying 
\[
       \fder{u^{(k)}}{\nu}=\sigma_k u^{(k)} 
       \quad \mbox{ on } \quad \partial M_1 \sqcup \partial M_2,
\]
with $\sigma_k =\lim_{i \rightarrow \infty} \sigma_k(M_{\rho_i})$. By an argument as in the proof of Theorem \ref{theorem:gluing-boundary}, $u^{(k)}$ extends to a harmonic function on $M_1 \sqcup M_2$, and hence to a Steklov eigenfunction with eigenvalue $\sigma_k$ on $M_1 \sqcup M_2$.

We now show that $u^{(k)}$ is a $k$-th eigenfunction of $M_1 \sqcup M_2$; i.e. $\sigma_k=\sigma_k(M_1 \sqcup M_2)$. We prove this by induction on $k$. First, since $\sigma_0(M_\rho)=0$, we have that $\sigma_0= \lim_{\rho \rightarrow 0} \sigma_0(M_\rho)=0$, and so $\sigma_0=\sigma_0(M_1 \sqcup M_2)$. Now suppose $\sigma_l=\sigma_l(M_1 \sqcup M_2)$ for $l=1, \ldots, k-1$, where $k \geq 1$. We will show that $\sigma_k=\sigma_k(M_1 \sqcup M_2)$. First observe that $\{ u^{(k)}\}_{k=1}^\infty$ are $L^2$-orthonormal on $\partial (M_1 \sqcup M_2)$, since $\{u^{(k)}_{\rho_i}\}_{k=1}^\infty$ are $L^2$-orthonormal on $\partial M_{\rho_i}=\partial (M_1 \sqcup M_2)$. It follows that $\sigma_k \geq \sigma_k(M_1 \sqcup M_2)$. The proof that $\sigma_k \leq \sigma_k(M_1 \sqcup M_2)$ follows exactly as in the proof of Theorem \ref{theorem:gluing-boundary}. Therefore, $\lim_{i \rightarrow \infty} \sigma_k(M_{\rho_i})=\sigma_k(M_1 \sqcup M_2)$.
\end{proof}

\begin{remark}
The same spectral convergence result holds for more complicated gluing constructions along the interior of manifolds. Specifically, the geometry of the neck region does not affect the spectrum in the limit. All that is needed in the proof of Theorem \ref{theorem:gluing-interior} is that as $\rho \rightarrow 0$, $M_\rho \setminus \mbox{Int} \, T_\rho$ converges to $(M_1 \setminus S_1, g_1) \sqcup (M_2 \setminus S_2, g_2)$, where $S_i \subset \mbox{Int} \, M_i$ is a set of Hausdorff dimension at most $n-2$, for $i=1, \, 2$. In this case a similar removable singularity argument shows that $u^{(k)}$ extends from a harmonic function on $M_i \setminus S_i$ to a smooth harmonic function on $M_i$, for $i=1, \, 2$. The rest of the proof carries through unchanged.
\end{remark}

The same argument can be used to glue a single manifold to itself at distinct interior points.

\begin{theorem} \label{theorem:gluing-single-interior}
Let $M$ be an $n$-dimensional Riemannian manifold with nonempty boundary. Given any $\epsilon >0$ there exists a manifold $M_\epsilon$ obtained by appropriately gluing $M$ to itself near distinct interior points, such that $\partial M_\epsilon  = \partial M$ and
\[
     \lim_{ \epsilon \rightarrow 0} \sigma_k(M_\epsilon) = \sigma_k(M)  
\]
for $k=0, \, 1,\, 2, \ldots$.
\end{theorem} 

We close this section by mentioning an immediate application of the continuity of the first $k$ Steklov eigenvalues under certain degenerations for surfaces. Given an orientable surface $M$ of genus $\gamma$ with $m$ boundary components, let
\[
       \sigma^*_k(\gamma, m)=\sup \{ \sigma_k(M,g)L_g(\partial M)  :  g \mbox{ a smooth metric on } M \}.
\]
For any surface, there is an upper bound 
\[
      \sigma^*_k(\gamma,m) \leq 2\pi(\gamma + m + k -1)
\]
independent of the metric (\cite{K1}). However, the exact value of $\sigma^*_k(\gamma,m)$ is only known in a few cases. As discussed in section \ref{section:disk}, $\sigma^*_k(0,1)=2\pi k$ (\cite{W}, \cite{HPS}, \cite{GP1}), and is achieved by the Euclidean disk for $k=1$, but is not achieved for any $k \geq 2$ (Theorem \ref{theorem:disk}, and \cite{GP1} for $k=2$). The only other sharp upper bounds that are known are for $k=1$ for the annulus and M\"obius band. In \cite{FS4} the authors proved that $\sigma^*_1(0,2)=4\pi/T_{1,0}$ where $T_{1,0} \approx 1.2$ is the unique positive solution of $t=\coth t$, and the supremum is uniquely (up to $\sigma$-homothety) achieved by the induced metric on the critical catenoid. 

As a consequence of the gluing results of this section, we have the following lower bound for $\sigma^*_k(\gamma, m)$, as discussed in \cite[Equation (0.2)]{P2}.

\begin{corollary} \label{pet_cor}
\[
      \sigma^*_k(\gamma,m) \geq 
      \max_{\substack{k_1+ \cdots + k_s=k \\ 
                     k_j \geq 1 \, \forall j \\ 
                     \gamma_1 + \cdots + \gamma_s \leq \gamma \\
                     m_1 + \cdots m_s + \gamma_1 + \cdots + \gamma_s \leq m + \gamma \\
                     \gamma_1 < \gamma {\tiny{\mbox{ or }}} m_1 + \gamma_1 < m + \gamma \mbox{\tiny{  if }} s=1}}
      \sum_{j=1}^s \sigma^*_{k_j} (\gamma_j,m_j)
\]
\end{corollary}

\begin{proof}
Suppose the maximum of the right hand side is achieved for some  $k_1, \ldots, k_s$, $\gamma_1, \ldots, \gamma_s$, and $m_1, \ldots , m_s$. Let $M_{\gamma_j, m_j}$ be a Riemannian surface of genus $\gamma_j$ with $m_j$ boundary components such that $\overline{\sigma}_{k_j}(M_{\gamma_j, m_j})$ is arbitrarily close to $\sigma^*_{k_j}(\gamma_j,m_j)$. By rescaling the metrics on the surfaces we may assume that $\sigma_{k_j}(M_{\gamma_j, m_j})=1$ for $j=1, \ldots, s$. Then $\sigma_k(M_{\gamma_1, m_1} \sqcup \cdots \sqcup M_{\gamma_s,m_s})=1$ and $\overline{\sigma}_k(M_{\gamma_1, m_1} \sqcup \cdots \sqcup M_{\gamma_s,m_s})=\sum_{j=1}^s \overline{\sigma}_{k_j}(M_{\gamma_j,m_j})$ which is arbitrarily close to $\sum_{j=1}^s \sigma^*_{k_j}(\gamma_j,m_j)$. Using Theorem \ref{theorem:gluing-interior} we glue the surfaces $M_{\gamma_1,m_1}, \ldots, M_{\gamma_s,m_s}$ together using cylindrical necks between interior points to obtain a Riemannian surface $M$ of genus $\gamma_1+ \cdots + \gamma_s$ with $m_1 + \cdots + m_s$ boundary components, and such that $\overline{\sigma}_k(M)$ is arbitrarily close to $\overline{\sigma}_k(M_{\gamma_1, m_1} \sqcup \cdots \sqcup M_{\gamma_s,m_s})$. If $m-(m_1 + \cdots + m_s)=l>0$, then using Theorem \ref{theorem:gluing-single-boundary} we glue two of the boundary components of $M$ together to reduce the number of boundary components by one and increase the genus by one, while changing the normalized eigenvalues by an arbitrarily small amount. Doing this $l$ times, we obtain a surface with $m$ boundary components and genus $\gamma_1 + \cdots + \gamma_s + l \leq \gamma$. On the other hand, if $m-(m_1 + \cdots + m_s)=l<0$, then we remove $l$ small disjoint disks from $M$ to obtain a surface with $m$ boundary components with genus $\gamma_1 + \cdots + \gamma_s \leq \gamma$, while changing the normalized eigenvalues by an arbitrarily small amount (\cite[Proposition 4.3]{FS4}). In either case, if the resulting surface has genus less than $\gamma$, then using Theorem \ref{theorem:gluing-single-interior} we glue the surface to itself between two interior points to increase the genus by one without changing the number of boundary components, while changing the normalized eigenvalues by an arbitrarily small amount. Repeating this as necessary, we obtain a Riemannian surface $M'$ of genus $\gamma$ with $m$ boundary components with $\overline{\sigma}_k(M')$ arbitrarily close to $\sum_{j=1}^s \sigma^*_{k_j}(\gamma_j,m_j)$.
\end{proof}

\section{Higher Steklov eigenvalues for the annulus and M\"obius band}

It is an open question to determine the suprema of the higher Steklov eigenvalues among all smooth metrics on the annulus and M\"obius band, and whether the suprema are achieved. For the first nonzero eigenvalue,
as discussed in Section \ref{section:rigidity}, the authors proved in \cite{FS4} that there exists a smooth metric on the annulus and on the M\"obius band that maximizes the first nonzero normalized Steklov eigenvalue, and explicitly characterized the maximizing metric as the induced metric on the critical catenoid and the critical M\"obius band, respectively. The characterization of the maximizing metrics involves a nontrivial argument showing that a metric that maximizes the first nonzero eigenvalue must be $\sigma$-homothetic to an $S^1$-invariant metric. The result then follows from an analysis of $S^1$-invariant metrics on the annulus \cite[Section 3]{FS1} and M\"obius band \cite[Proposition 7.1]{FS4}. In particular, {\em the supremum of the first nonzero eigenvalue over all metrics is the same as the supremum of the first nonzero eigenvalue among all $S^1$-invariant metrics}. One can then ask whether anything like this is true for the higher eigenvalues. \cite{FTY} and \cite{FSa} extended the analysis of $S^1$-invariant metrics to higher Steklov eigenvalues, and for each $k \geq 2$, determined the supremum of the $k$-th nonzero normalized Steklov eigenvalue among all $S^1$-invariant metrics on the annulus and the M\"obius band. Moreover, in each case, the supremum is achieved by the induced metric on an explicit free boundary annulus or M\"obius band in a Euclidean ball, except for the supremum of second normalized eigenvalue on the annulus, which is not achieved. In summary, in the case of the annulus, Fan-Tam-Yu proved:

\begin{theorem}[\cite{FTY}] \label{theorem:FTY}
Let $\sigma^{S^1}_k$ be the supremum of $k$-th normalized Steklov eigenvalue among all $S^1$-invariant metrics on the annulus. 

\vspace{1mm}

(i) $\sigma^{S^1}_2=4\pi$. Moreover, $\bar{\sigma}_2(g_T) \rightarrow 4\pi$ as $T \rightarrow \infty$, where $g_T=dt^2+d\theta^2$ on the cylinder $[0,T]\times \mathbb{S}^1$, and the supremum $4\pi$ is not achieved. 

\vspace{1mm}

(ii) $\sigma^{S^1}_{2k-1}=4k\pi/T_{1,0}$ for all $k \geq 1$, where $T_{1,0}$ is the unique positive solution of $t=\coth t$, and is achieved by the induced metric on the $k$-critical catenoid. 

\vspace{1mm}

(iii) $\sigma_{2k}^{S^1}=4k\pi \tanh (kT_{k,1})$ for $k>1$, where $T_{k,1}$ is the unique positive solution of 
$k\tanh(kt)=\coth(t)$, and is achieved by the induced metric from an explicit free boundary minimal immersion of the annulus into $\mathbb{B}^4$.
\end{theorem}

Here we use the notation $\bar{\sigma}_k(g):=\sigma_k(g) L_g(\partial M)$ for the $k$-th normalized Steklov eigenvalue of a surface $(M,g)$.
In the case of the M\"obius band, Fraser-Sargent proved:

\begin{theorem}[\cite{FSa}] \label{theorem:FSa}
Let $\sigma_k^{S^1}$ be the supremum of the $k$-th normalized Steklov eigenvalue among $S^1$-invariant metrics on the M\"obius band. For all $k \geq 1$, 
\[
      \sigma_{2k-1}^{S^1}=\sigma^{S^1}_{2k}=4\pi k \tanh(2kT_{2k,1})
\]
and the supremum is achieved by the induced metric from an explicit free boundary minimal embedding of the M\"obius band into $\mathbb{B}^4$.
\end{theorem}
     
It is natural to ask whether the maximizers for the higher eigenvalues among $S^1$-invariant metrics, in Theorem \ref{theorem:FTY} on the annulus and Theorem \ref{theorem:FSa} on the M\"obius band, also maximize among {\em all} metrics, as they do for the first eigenvalue when $k=1$. We show that this is not the case for the higher eigenvalues. Specifically, for $k \geq 2$, using Theorem \ref{theorem:gluing-boundary} we construct smooth metrics on the annulus and M\"obius band with $k$-th eigenvalue strictly bigger than the supremum of the $k$-th eigenvalue over $S^1$-invariant metrics. 

\begin{theorem} \label{theorem:noninvariant}
For $k \geq 2$, the supremum $\sigma^*_k$ of the $k$-th normalized Steklov eigenvalue over all smooth metrics on the annulus (or respectively, M\"obius band) is strictly bigger than the supremum $\sigma^{S^1}_k$ 
over $S^1$-invariant metrics on the annulus (or respectively, M\"obius band).
\end{theorem}

\begin{proof}
Fix $k \geq 2$. Let $\tilde{M}$ be the disjoint union of the critical catenoid $C$ and $k-1$ Euclidean unit disks $\mathbb{D}$. Observe that 
\[
    \sigma_0(\tilde{M})=\sigma_1(\tilde{M})=\cdots = \sigma_{k-1}(\tilde{M})=0, \quad \sigma_k(\tilde{M})=1
\]       
and 
\[
      \bar{\sigma}_k(\tilde{M})=L(\partial C)+(k-1) L(\partial \mathbb{D}) 
      =\frac{4\pi}{T_{1,0}}+2(k-1)\pi > \frac{4\pi}{1.2}+2(k-1)\pi.
\]
By Theorem \ref{theorem:gluing-boundary}, for any $\epsilon>0$, there is a smooth metric annulus $M$ obtained by gluing $C$ and $k-1$ disks $\mathbb{D}$ together, such that $|\bar{\sigma}_k(\tilde{M})-\bar{\sigma}_k(M)| < \epsilon$. We claim that $\bar{\sigma}_k(M) > \sigma^{S^1}_k$, where $\sigma^{S^1}_k$ is the supremum of the $k$-th normalized Steklov eigenvalue over all $S^1$-invariant metrics on the annulus. First note that for $k=2$, 
\[
      \bar{\sigma}_2(\tilde{M}) > \frac{4\pi}{1.2}+2\pi > 4\pi =\sigma_2^{S^1}.
\]
For $k=2l-1$ odd with $l >1$, we have
\[
     \bar{\sigma}_k(\tilde{M})> \frac{4\pi}{1.2}+2(2l-2)\pi=\frac{4\pi l +.8\pi(l-1)}{1.2}
     > \frac{4\pi l}{1.2} > \frac{4\pi l }{T_{1,0}} = \sigma_k^{S^1}.
\]
For $k=2l$ even with $l>1$, we have
\[
     \bar{\sigma}_k(\tilde{M})> \frac{4\pi}{1.2}+2(2l-1)\pi > 4l\pi > 4l\pi \tanh(lT_{k,1}) =\sigma_k^{S^1}.
\]
For each $k \geq 2$, by choosing $\epsilon>0$ sufficiently small, it follows that $\bar{\sigma}_k(M)>\sigma_k^{S^1}$.

We now consider the case of the M\"obius band. In this case, we let $\tilde{M}$ be the disjoint union of the critical M\"obius band $C$ and $k-1$ Euclidean disks $\mathbb{D}$. Observe that
\[
    \sigma_0(\tilde{M})=\sigma_1(\tilde{M})=\cdots = \sigma_{k-1}(\tilde{M})=0, \quad \sigma_k(\tilde{M})=1
\]       
and 
\[
      \bar{\sigma}_k(\tilde{M})=L(\partial C)+(k-1) L(\partial \mathbb{D}) =2\pi\sqrt{3}+2(k-1)\pi.
\]
By Theorem \ref{theorem:gluing-boundary}, for any $\epsilon>0$, there is a smooth metric M\"obius band $M$ obtained by gluing $C$ and $k-1$ disks $\mathbb{D}$ together, such that $|\bar{\sigma}_k(\tilde{M})-\bar{\sigma}_k(M)| < \epsilon$. We claim that $\bar{\sigma}_k(M) > \sigma^{S^1}_k$, where now $\sigma_k^{S^1}$ denotes the supremum of the $k$-th normalized Steklov eigenvalue over all $S^1$-invariant metrics on the M\"obius band. For $k=2l$ even, with $l \geq1$, this is clear, since
\[
     \bar{\sigma}_k(\tilde{M})=2\pi\sqrt{3}+2(2l-1)\pi > 4\pi l > 4\pi l \tanh(2l T_{2l,1}) =\sigma_k^{S^1}.
\]
For $k=2l-1$ odd, with $l >1$, we need a better approximation of $4\pi l \tanh(2l T_{2l,1})$. First observe that\[
      \frac{d}{dt} \coth t =-\frac{1}{\sinh^2 t} > -\frac{1.2}{t^2}= \frac{d}{dt} \left(\frac{1.2}{t}\right)
\]
since $\sinh t / t \geq 1$ for all $t$. Also, $T_{2,1}=\ln(2+\sqrt{3})/2$, and $\coth(T_{2,1})=\sqrt{3} < 1.2/T_{2,1}$.
It follows that $\coth t < 1.2/t$ for all $t < T_{2,1}$. Denote by $t_k$ the unique positive solution of $k\tanh (kt)=1.2/t$. Recall that $T_{k,1}$ is the unique positive solution of $k\tanh(kt)=\coth t$. Since $\coth t < 1.2/t$ for all $t < T_{2,1}$, $T_{k,1}<T_{2.1}$ for $k > 2$ (\cite[Lemma 2.3]{FTY}), and $k \tanh(kt)$ is increasing in $t$, it follows that $T_{k,1} < t_k$. Therefore, if $k >2$,
\[
      k\tanh(kT_{k,1})<k\tanh(kt_k). 
\]
By definition of $t_k$, $k\tanh(k t_k)={1.2}/{t_k}$.
Therefore $\tanh(k t_k)=1.2/(kt_k)$ and so $t_1=kt_k$. By approximation we have that $t_1>1.36$. Finally, for $l >1$,
\[
     \sigma_{2l-1}^{S^1}=4\pi l \tanh(2lT_{2l,1})<4\pi l \tanh(2lt_{2l})
     =2\pi \frac{1.2}{t_{2l}}=4\pi l\frac{1.2}{t_1} < 4\pi l\frac{1.2}{1.36} < 2\pi l \cdot (1.77).
\]
On the other hand, 
\[
      \bar{\sigma}_{2l-1}(\tilde{M})=2\pi\sqrt{3}+2(2l-2)\pi =2\pi ( 2 l + \sqrt{3} -2).
\]
If $l >1$, it is straightforward to check that $2 l + \sqrt{3} -2 > 1.77 \, l$, and so 
$\bar{\sigma}_{2l-1}(\tilde{M}) > \sigma_{2l-1}^{S^1}$. 
For each $k \geq 2$, if $\epsilon>0$ is sufficiently small, then $\bar{\sigma}_k(M)>\sigma_k^{S^1}$.
\end{proof}

\begin{remark}
As in the case of the disk, it might be reasonable to expect that maximizing metrics do not exist for higher eigenvalues on the annulus and M\"obius band, and to ask: 

\begin{enumerate}
\item[(i)]
Is the supremum of the $k$-th nonzero normalized Steklov eigenvalue among all smooth metrics on the annulus $4\pi/T_{1,0}+2(k-1)\pi$, where $T_{1,0} \approx 1.2$ is the unique positive number such that $\coth t =t$?

\item[(ii)]
Is the supremum of the $k$-th nonzero normalized Steklov eigenvalue among all smooth metrics on the M\"obius band $2\pi\sqrt{3} + 2(k-1)\pi$?
\end{enumerate}
\end{remark}
That this might be true is also suggested by results for higher eigenvalues of the Laplacian on the two-sphere and real projective plane \cite{KNPP}, \cite{K3}, \cite{P1}, \cite{N}, \cite{NS}, \cite{NP}.
     
\bibliographystyle{plain}

\end{document}